\documentclass[10pt,a4paper]{amsart}

\usepackage[latin1]{inputenc}
\usepackage{amsmath}
\usepackage{amsfonts}
\usepackage{amssymb}
\usepackage{amsthm}
\usepackage{enumerate}
\usepackage{mathrsfs}
\usepackage{cite}
\usepackage{verbatim}
\usepackage{a4wide}
\usepackage{dsfont}
\newcommand{\bbb}[1][\mathds{1}]{\mathds{#1}}
\newcommand{\mbold}[1]{\mathbb{#1}}

\newtheoremstyle{primedtheorem}{}{}{\itshape}{}{\scshape}{}{3pt}{#2 {#1}'.}
\newtheoremstyle{swappedplain}{}{}{\itshape}{}{\scshape}{}{3pt}{#2 #3 #1.}
\newtheoremstyle{swappedroman}{}{}{\rmfamily}{}{\scshape}{}{3pt}{#2 #3 #1.}
\newtheoremstyle{swappedromancontents}{}{}{\rmfamily}{}{\scshape}{}{3pt}{#2 #3 #1. \addcontentsline{toc}{subsubsection}{#2 #3}}
\newtheoremstyle{questionstyle}{}{}{\rmfamily}{}{\scshape}{}{3pt}{* #1 #2 *}

\newtheorem{theorem}{Theorem}
\newtheorem*{theorem*}{Theorem}
\newtheorem{lemma}[theorem]{Lemma}
\newtheorem*{lemma*}{Lemma}
\newtheorem{proposition}[theorem]{Proposition}
\newtheorem*{proposition*}{Proposition}

\newtheorem{corollary}[theorem]{Corollary}
\newtheorem*{corollary*}{Corollary}

\newtheorem*{definition*}{Definition}

\newtheorem*{notation*}{Notation}

\newtheorem*{remark*}{Remark}

\newtheorem*{example*}{Example}

\newtheorem*{observation*}{Observation}
\newtheorem{conjecture}[theorem]{Conjecture}
\newtheorem*{conjecture*}{Conjecture}

\newcommand{\supp}{\operatorname{supp}}

\newcommand{\todo}[1]{}

\address{Nicolas Monod, SB-MATHGEOM-EGG, EPFL, Station 8, CH-1015 Lausanne, Switzerland.}
\email{nicolas.monod@epfl.ch\kern5mm{\it URL}:\;egg.epfl.ch/$\sim$nmonod}
\author{Nicolas Monod}

\address{Henrik Densing Petersen,
Department of Mathematical Sciences\\
University of Copenhagen\\
Universitetsparken 5
2100 K{\o}benhavn \O, 
Denmark.}
\email{hdp@math.ku.dk\kern5mm{\it URL}:\;www.math.ku.dk/$\sim$hdp}
\thanks{NM was supported in part by the Swiss National Science Foundation and the ERC; HDP acknowledges the support of the Danish National Research Foundation (DNRF) through the Centre for Symmetry and Deformation.}

\author{Henrik D. Petersen}

\title[An obstruction to $\ell^p$-dimension]{An obstruction to $\boldsymbol\ell^{\boldsymbol p}$-dimension}

\begin{document}
\begin{abstract}
For any group $G$ containing an infinite elementary amenable subgroup, and any $2<p<\infty$, there exists closed invariant subspaces $E_i\nearrow \ell^pG$ and $F\neq 0$ such that $E_i\cap F = 0$ for all $i$. This is an obstacle to $\ell^p$-dimension and gives a negative answer to a question of Gaboriau.
\end{abstract}

\maketitle

\section*{Introduction and Statement of the Result}
The Murray--von Neumann dimension, also called \emph{$\ell^2$-dimension}, has proved tremendously valuable since its introduction in 1936~\cite{MvN}. This is particularly true for its use through \emph{$\ell^2$-Betti numbers} of groups since the seminal article by Cheeger and Gromov~\cite{CG86}. This has prompted recent work in pursuit of a more general $\ell^p$-dimension for $1 < p < \infty$; see e.g.~\cite{Go10,Go12,Ben11}.  The purpose of this short note is to establish a fundamental obstruction to this endeavour.

\medskip
%In~\cite[Theorem 8.6]{Ben11}, Hayes computes, as a test case, ``$\ell^p$-Betti numbers'' of the free (non-abelian) groups for $p \leq 2$, defined in analogy with $\ell^2$-Betti numbers using his $\ell^p$-dimension for sofic groups, extending the well-known computation in the case $p=2$~\cite{CG86}.
%
%\smallskip
%
One of the central motivations to define a notion of $\ell^p$-dimension is to take advantage of the situations where $\ell^p$-cohomology carries information not accessible through to $\ell^2$-methods, see e.g.~\cite{Pansu08}. However, this involves typically $p$ large, in particular $p > 2$; unfortunately this range appears to be much more difficult than $p\leq 2$ (compare e.g.~\cite[Theorem 8.6]{Ben11}). 

\medskip
Given a (discrete) group $G$, a concrete test for the existence of an $\ell^p$-dimension with at least some positivity and continuity properties is the following question of Gaboriau, see Question~1.1 in~\cite{Go12}. Let $(E_i)_{i\in \mbold{N}}$ be an \emph{invariant exhaustion} of $\ell^p G$, i.e.\ an increasing sequence of (left-) $G$-invariant closed subspaces $E_i\subseteq \ell^p G$ with dense union. If a $G$-invariant closed subspace $F\subseteq \ell^p G$ meets each $E_i$ trivially, does it follow that $F$ is trivial?

This question also seems to be the main missing point to establish the measure-equivalence invariance of the vanishing of $\ell^p$-cohomology in parallel to~\cite{Ga02}. Such invariance would notably settle the well-known open problem of the vanishing of reduced $\ell^p$-cohomology for amenable groups.

\medskip
The case of \emph{amenable} groups was expected to be much more tractable, and indeed Gournay~\cite{Go12} has given a positive answer to Gaboriau's question when $G$ is amenable and $p\leq 2$. The present note shows that the situation is very much opposite for $p>2$, even when $G=\mbold{Z}$.

\begin{theorem}\label{thm:main}
Let $G$ be an infinite elementary amenable group and let $2<p<\infty$.

There exists a closed invariant subspace $0 \neq F\subseteq \ell^pG$ and an invariant exhaustion $(E_i)_{i\in \mbold{N}}$ of $\ell^pG$ such that $E_i\cap F = 0$ for all $i\in \mbold{N}$.
\end{theorem}

Our current proof of this theorem involves notably a cameo appearance of the Feit--Thompson Theorem~\cite{Feit-Thompson} and a result of Saeki~\cite{Sa80} in classical harmonic analysis.

\medskip
In general, we will say that an invariant exhaustion $(E_i)_i$ of $\ell^pG$ is \emph{thin}, if there is a closed invariant subspace $F\neq 0$ satisfying $E_i\cap F=0$ for all $i$. It is not hard to see (and will be needed in the proof anyway) that the existence of such exhaustions can be induced up from subgroups. We thus record the following formally stronger statement:

\begin{corollary}\label{cor:main}
Thin invariant exhaustions of $\ell^p G$ for all $p>2$ exist more generally for every group $G$ containing an infinite elementary amenable subgroup.

\nobreak
In particular, this holds for every non-trivial group which is not torsion.
\end{corollary}

Thus, for instance, any infinite linear group $G$ over any field of any characteristic admits thin invariant exhaustions of $\ell^p G$ for all $p>2$ because the assumptions of Corollary~\ref{cor:main} are satisfied due to the Tits alternative (Theorems~1 and~2 in~\cite{Tits72}).

\bigskip
Perhaps one should find another proof of Theorem~\ref{thm:main}.  In any case, this would be necessary to address the following.

\begin{conjecture}
Any infinite group $G$ admits a thin invariant exhaustion of $\ell^p G$ for all $p>2$.
\end{conjecture}

\section*{Proof of the Result} %Theorem~\ref{thm:main} and Corollary~\ref{cor:main}}
First we observe that the property of \emph{not} having a thin invariant exhaustion is hereditary:

\begin{lemma} \label{lma:hereditary}
Let $G$ be a group and $1\leq p<\infty$. Suppose that $G$ contains a subgroup $H$ which has a thin invariant exhaustion $(E_i)_i$ of $\ell^pH$. Then $G$ has a thin invariant exhaustion of $\ell^pG$.
\end{lemma}

Thus Corollary~\ref{cor:main} follows indeed from Theorem~\ref{thm:main}.

\begin{proof}[Proof of Lemma~\ref{lma:hereditary}]
Let $G$ and $p$ be given and fix a subgroup $H$ of $G$ and a thin invariant exhaustion $(E_i)_i$ of $\ell^pH$.  We have a $G$-equivariant isomorphism of $\ell^p$-spaces
\begin{equation*}
\ell^p G \cong \ell^p(G/H,\ell^p(H)),
\end{equation*}
where the $G$-action on the right-hand side is given by $(g.\xi)(x) = c(g,x).\xi(g^{-1}.x)$ for $g\in G, x\in G/H, \xi\in \ell^p(G/H,\ell^p(H))$, and $c:G\times G/H \rightarrow H$ a cocycle representative for the inclusion $H < G$.

Now take $F\neq0$ a closed invariant subspace of $\ell^pH$ such that $E_i\cap F=0$ for all $i$, and define $\tilde{E}_i := \ell^p(G/H,E_i)$ and similarly $\tilde{F}:=\ell^p(G/H,F)$. Then $(\tilde{E}_i)_i$ is a thin invariant exhaustion of $\ell^pG$.
\end{proof}

From here on, the proof splits into two cases:

\begin{enumerate}[(A)]
\item $G$ contains an element of infinite order.
\item $G$ is a torsion group.
\end{enumerate}

\subsection*{Proof in case (A)}
By Lemma~\ref{lma:hereditary}, we can assume $G\cong \mbold{Z}$. Recall that if $\mu$ is a measure on the Borel $\sigma$-algebra in a topological space $X$, the support of $\mu$ is defined by
\begin{equation*}
\supp \mu := \{ x\in X \mid \mu(V)>0\ \ \forall\, V\owns x \; \textrm{open} \}.
\end{equation*}
It follows directly from the definition that $\supp \mu$ is a closed set. We denote the Lebesgue measure on $\mbold{T}$ by $\lambda$ and the Fourier--Stieltjes coefficients of $\mu$ by $\widehat\mu$.

\begin{proposition} \label{prop:caseZ}
Let $2<p<\infty$ and suppose that $\mu$ is a measure on $\mbold{T}$ such that $\lambda(\supp \mu) = 0$ and $\widehat{\mu} \in \ell^p\mbold{Z}$. Denote by $F$ the closed, invariant subspace of $\ell^p\mbold{Z}$ generated by $\widehat{\mu}$. Then there is an invariant exhaustion $(E_i)_{i\in \mbold{N}}$ of $\ell^p\mbold{Z}$ such that $F\cap E_i=0$ for all $i \in \mbold{N}$.
\end{proposition}

\begin{proof}
Let $p,\mu$ be given as in the statement and denote $K:=\supp \mu$. We can choose a decreasing family $(U_i)_{i\in \mbold{N}}$ of open neighbourhoods of $K$, such that
\begin{enumerate}[(i)]
\item $K=\bigcap_i U_i$.
\item $\overline{U_{i+1}} \subseteq U_{i}$ for all $i\in \mbold{N}$.
\end{enumerate}

By the existence of a smooth partition of unity, there exist smooth functions $\sigma_i$ on $\mbold{T}$ such that for every $i\in \mbold{N}, t\in \mbold{T}$ we have
\begin{equation*}
\sigma_i(t) =
\begin{cases}
0 & \mbox{if } t\in \overline{U_{i+1}},\\
1 & \mbox{if } t\in U_i^{\complement}.
\end{cases}
\end{equation*}
On $U_i\setminus \overline{U_{i+1}}$, we do not care about the value of $\sigma_i(t)$ as long as $\sigma_i$ is smooth.

By the smoothness of $\sigma_i$, the Fourier transforms $S_i:=\widehat{\sigma_i}$ are in $\ell^1\mbold{Z}$ for all $i\in \mbold{N}$. Further, we observe that in $\ell^1\mbold{Z}$ we have for every $i\in \mbold{N}$ that 
\begin{equation} \label{eq:Sid}
S_{i+1}*S_i=S_i,
\end{equation}
since the identity $\sigma_{i+1}(t)\cdot \sigma_i(t) = \sigma_i(t)$ holds pointwise on $\mbold{T}$. Indeed, fixing some $i$ we observe that both sides of the latter equality are zero if $t\in U_{i+1}$, and that if $t\in U_{i+1}^{\complement}$ then $\sigma_{i+1}(t)=1$ so that the equality is tautological.

From this point on we consider the $S_i$ as translation-invariant operators on $\ell^p\mbold{Z}$, acting by convolution.% Since the representation of $\ell^1\mbold{Z}$ on $\ell^p\mbold{Z}$ is faithful, the identity~\eqref{eq:Sid} still holds as operators on $\ell^p\mbold{Z}$.

Denote by $F$ the closed invariant subspace of $\ell^p\mbold{Z}$ generated $\widehat{\mu}$. By~\cite[Lemma 2]{Pu98} we have $F\subseteq \ker S_i$ for all $i$. We put $E_i := \ker(\operatorname{Id}_{\ell^p\mbold{Z}}-S_i)$. The $E_i$ are then closed, invariant subspaces, and it is obvious that $F\cap E_i=0$ for all $i\in \mbold{N}$.

The fact that $E_i\subseteq E_{i+1}$ for all $i\in \mbold{N}$ follows directly from the identity~\eqref{eq:Sid}. Indeed, take $x\in E_i$. Then we get $x=S_i.x=(S_{i+1}S_i).x=S_{i+1}.(S_i.x) = S_{i+1}.x$ whence $x\in E_{i+1}$.

Finally, we need to see that $\bigcup_iE_i$ is dense in $\ell^p\mbold{Z}$. For this denote by $\iota_p:\ell^2\mbold{Z}\rightarrow \ell^p\mbold{Z}$ the canonical embedding. For each $i\in \mbold{N}$ there is an orthogonal projection $p_i\in L\mbold{Z}$, the von Neumann algebra generated by the left-regular representation of $\mbold{Z}$ on $\ell^2\mbold{Z}$, such that under that spatial isomorphism $L\mbold{Z} \cong L^{\infty}\mbold{T}$, $p_i$ corresponds to the indicator function, $\bbb_{U_i^{\complement}}$, of the complement of $U_i$. Then, denoting by $S_i^{(2)}$ the operator of convolution by $S_i$ on $\ell^2\mbold{Z}$ we have $p_i(\ell^2\mbold{Z}) \subseteq \ker(\operatorname{Id}_{\ell^2\mbold{Z}} -S_i^{(2)})$. It follows that $\iota_p(p_i(\ell^2\mbold{Z})) \subseteq \ker(\operatorname{Id}_{\ell^p\mbold{Z}} - S_i) = E_i$.

Then it just remains to note that, since $\lambda(U_i^{\complement}) \nearrow_i 1$, the union $\bigcup_i p_i(\ell^2\mbold{Z})$ is dense in $\ell^2\mbold{Z}$. Thus $\bigcup_i \iota_p(p_i(\ell^2\mbold{Z}))$ is dense in $\ell^p\mbold{Z}$.
\end{proof}

To complete the proof of Theorem~\ref{thm:main} in case (A), we now need to know that a measure $\mu$ as in the statement of the previous proposition does indeed exists. This follows from the main theorem of~\cite{Sa80} as we shall now recall. Let $\phi: c_0(\mbold{Z})\rightarrow c_0(\mbold{Z})$ be a continuous map (with respect to sup-norm); for instance, $\phi$ can be the composition by any continuous map $\phi_0:\mbold{R}\to\mbold{R}$ with $\phi_0(0)=0$. Then Saeki proves in~\cite{Sa80} that there is a probability measure $\mu$ with Lebesgue-null support such that the Fourier--Stieltjes coefficients $\widehat\mu(n)$ satisfy
\begin{equation}\label{eq:Saeki}
\sum_{n\in\mbold{Z}} \big| \widehat\mu(n)^2 \phi(\widehat\mu)(n) \big| <\infty.
\end{equation}
Thus, implicitly, $\widehat\mu$ vanishes at infinity; see Remark~(\ref{rem:c0}) below. If we choose for instance $\phi_0(x) = |x|^{p-2}$ for $p>2$, then~\eqref{eq:Saeki} shows that $\widehat\mu$ is in $\ell^p$, as desired. We can also obtain $\widehat\mu$ in all $\ell^p$ simultaneously if we arrange $\lim_{x\to 0} |x|^{\epsilon} /\phi_0(x) =0$ for all $\epsilon>0$; an explicit example is $\phi_0(x)=\exp(-\sqrt{\big|\log|x|\big|})$. This finishes the proof of case (A).

\subsection*{Proof in case (B)}
Since $G$ is an elementary amenable torsion group, it is locally finite by Theorem~2.3 in~\cite{Chou80}.  We can now appeal to a result of Hall--Kulatilaka~\cite{HaKu64}, relying in turn on the Feit--Thompson Theorem~\cite{Feit-Thompson}, to conclude that $G$ contains an infinite abelian subgroup. Therefore, using again Lemma~\ref{lma:hereditary}, we can assume that $G$ is a countably infinite locally finite abelian group. The proof of Theorem~\ref{thm:main} in this case follows the overall strategy of the case of $\mbold{Z}$, if one redefines a ``smooth function'' to mean a locally constant function on a totally disconnected compact space. We shall also need to address the question of the existence of a suitable measure $\mu$ on the Pontryagin dual $\widehat G$. We first prove a version of Proposition~\ref{prop:caseZ} in this setup. Let $\lambda_{\widehat{G}}$ denote the normalized Haar measure on the profinite group $\widehat{G}$.

\begin{proposition}\label{prop:caseTor}
Let $2<p<\infty$ and suppose that $\mu$ is a measure on $\widehat{G}$ such that $\lambda_{\widehat{G}}(\supp \mu) = 0$ and $\widehat{\mu} \in \ell^p G$. Denote by $F$ the closed, invariant subspace of $\ell^pG$ generated by $\widehat{\mu}$. Then there is an invariant exhaustion $(E_i)_{i\in \mbold{N}}$ of $\ell^pG$ such that $F\cap E_i=0$ for all $i \in \mbold{N}$.
\end{proposition}

\begin{proof}
We first observe that Lemma~2 of~\cite{Pu98} holds in general for $G$ countable abelian and the proof goes through verbatim in this generality. This forms the basis for imitating the proof of proposition~\ref{prop:caseZ}; the only point that needs justification is the ``smoothness'' of $\sigma_i$ to ensure $\widehat{\sigma_i}\in\ell^1 G$. In fact, we claim that we can even arrange $\widehat{\sigma_i}\in \mbold{C}G$.

Indeed, since $\widehat{G}$ is profinite, we can take $U_i$ to be a compact-open subset given as the pre-image $\pi^{-1}(V_i)$ of a subset $V_i\subseteq \hat{F}$ in some finite quotient $\pi: \widehat{G}\to \widehat{F}$, where $F$ is a finite subgroup of $G$. The indicator function $\bbb_{U_i^{\complement}}$ of the complement is $\bbb_{V_i^{\complement}}\circ \pi$ and hence $\widehat{\bbb_{U_i^{\complement}}}$ is supported on $F$. Therefore, defining $\sigma_i= \bbb_{U_i^{\complement}}$, the transform $\widehat{\sigma_i}$ is finitely supported. Now~\eqref{eq:Sid} holds and the rest of the proof is unchanged.
\end{proof}

We now need again to show that a measure $\mu$ as in Proposition~\ref{prop:caseTor} exists to finish the proof of the theorem. Saeki states in~\cite[Remark~(II)]{Sa80} that his method still works as long as $\widehat G$ does not contain an open subgroup of bounded order (in Saeki's notation, the desired conclusion is obtained by choosing $f$ to be the constant function one on our compact group $\widehat G$, which he denotes by $G$). He indicates however that a different technique is required when $\widehat G$ \emph{does} contain an open subgroup of bounded order, although that technique is not provided. We shall therefore propose a proof in the proposition below.

\begin{proposition}
If $\widehat G$ contains an open subgroup of bounded order, then it admits a probability measure $\mu$ such that $\lambda_{\widehat{G}}(\supp \mu) = 0$ and $\widehat{\mu} \in \ell^p G$ for all $p>2$.
\end{proposition}

\begin{proof}
The assumption allows us to apply Corollary~25.10 in~\cite{Hewitt-RossI} and deduce that $\widehat G$ is isomorphic to a (countably infinite) product of finite cyclic groups $\mbold{Z}/n_i\mbold{Z}$ where only finitely many distinct integers $n_i$ occur. Therefore, regrouping factors and using the Chinese remainder theorem, we can write $\widehat G\cong A_0 \times B^{\mbold{N}}$ for some finite abelian groups $A_0$ and $B\neq 0$. Let $k\geq 2$ be the cardinality of $B$. Regrouping factors some more, we write
$$\widehat G \cong A_0 \times \prod_{n=1}^\infty \Big(B^n\Big)^{k^n} = \prod_{j=0}^\infty A_j,$$
where for $j\geq 1$ each $A_j$ is a group of size $c_j$, a power of $k$, in such a way that $c_j=k^n$ occurs exactly $k^n$ times for each $n\in\mbold{N}$. We now define $\mu$ explicitly as the infinite product $\mu=\prod_{j=0}^\infty \mu_j$, where $\mu_0$ is the uniform distribution on $A_0$ and $\mu_j$ is the uniform distribution on the set $A_j\setminus \{0\}$ for all $j\geq 1$.

\medskip
Since $\lambda_{\widehat{G}}$ is the product of the normalized measures on each factor, the support of $\mu$ has Haar measure $\prod_{j= 1}^\infty (1-1/c_j)$. This product is zero because $\sum_{j= 1}^\infty 1/c_j =\infty$ since this is a sum of terms $k^{-n}$, each of which occurs $k^n$ times.

\medskip
We now fix $p>2$ and proceed to prove that $\widehat\mu$ is in $\ell^p G$. We have $G\cong \bigoplus_{j=0}^\infty \widehat{A_j}$ and the dual measure on $\widehat{A_j}$ is the counting measure. Therefore, it suffices to prove that the product of all $\ell^p$-norms of $\widehat{\mu_j}$ over $j\geq 1$ is finite. Writing $\bbb_X$ for the indicator function of a subset $X$ of $A_j$ or of $\widehat{A_j}$, we have $\widehat{\bbb_{\{0\}}} = c_j^{-1} \bbb_{\widehat{A_j}}$ for $0\in A_j$ and $\widehat{\bbb_{A_j}} = \bbb_{\{0\}}$ for $0\in \widehat{A_j}$. The density of $\mu_j$ can be written as $c_j(c_j-1)^{-1} (\bbb_{A_j} - \bbb_{\{0\}})$, which implies $\widehat{\mu_j} = \bbb_{\{0\}} - (c_j-1)^{-1} \bbb_{\widehat{A_j}\setminus \{0\}}$. We conclude that $\|\widehat{\mu_j}\|_p^p = 1 + (c_j-1)^{1-p}$. In other words, it remains to verify that the product $\prod_{j=1}^\infty ( 1 + (c_j-1)^{1-p})$ is finite, or equivalently that $\sum_{j=1}^\infty  (c_j-1)^{1-p}$ is finite. The latter series is $\sum_{n=1}^\infty k^n (k^n-1)^{1-p}$ which is indeed finite if and only if $p>2$.
\end{proof}

This completes the proof of Theorem~\ref{thm:main}.\qed

\section*{Remarks}
\begin{enumerate}[(i)]
\setlength{\itemsep}{10pt}
\item Saeki's measure $\mu$ on $\mbold{T}$ satisfies that $\mu*\mu$ has a continuous density. He points out that this cannot be achieved for certain groups in case~(B) above, see Remark~(III) of~\cite{Sa80}.

\item\label{rem:c0} Stating the above property~\eqref{eq:Saeki} of $\mu$ requires to know a priori that $\widehat\mu$ vanishes at infinity. This can be verified in Saeki's proof: it follows e.g.\ from condition~(1) on p.~230 in~\cite{Sa80}.

\item If $G$ is any amenable group, then a necessary condition for the existence of a thin invariant exhaustion $(E_i)_i$ as in Theorem~\ref{thm:main}, is that the invariant subspace $F$ have $\ell^p$-dimension zero in the sense of Gournay~\cite{Go12}. In fact, we conjecture that the converse is also true:

\begin{conjecture} \label{thm:localcriterion}
Let $G$ be a (countable, discrete) amenable group, $1\leq p < \infty$, and $F\subseteq \ell^pG$ a closed invariant subspace. Then $\dim_{\ell^p} F = 0$ if and only if there is an invariant exhaustion $(E_i)_i$ of $\ell^pG$ such that $E_i\cap F = 0$ for all $i$.

\nobreak
More generally one can also consider this conjecture for sofic groups~\cite{Ben11}.
\end{conjecture}

This conjecture is motivated by a result of Sauer~\cite[Theorem 2.4]{Sa05} for L{\"u}ck's extended dimension function.

\item For $G=\mbold{Z}^n, n\geq 2$ and $p>\frac{2n}{n-1}$ the situation in Theorem~\ref{thm:main} is particularly nice. In this case one can take the subspace $F$ to be the kernel of an operator $T \in \mbold{C}\mbold{Z}^n$ which acts without kernel on $\ell^2\mbold{Z}^n$, and it was noticed already by Gournay~\cite[Remark 9.4]{Go12} that in this case the kernel in $\ell^p\mbold{Z}^n$ must have $\ell^p$-dimension zero, in the sense of~\cite{Go12}.

For a proof that there exist a $T$ with these properties see~\cite{Pu98}. The proof that we can take the kernel of $T$ as our $F$ is analogous to the proof of Proposition~\ref{prop:caseZ}.

\item Using a Euler characteristic argument, Gaboriau had previously observed that for certain non-amenable groups $G$ one cannot hope to have a notion of $\ell^p$-dimension for which $(\ell^p G)^{\oplus n}$ has dimension $n$, and which also satisfies additivity for short exact sequences (see the introduction of~\cite{Go10}). The fact that we produce thin exhaustions for \emph{amenable} groups should further increase the doubts that there be any reasonable $\ell^p$-dimension for large $p$.

\end{enumerate}

\section*{Acknowledgements}
\vspace{-2mm}
Question~\cite[1.1]{Go12} was suggested to us by Kate Juschenko. We also thank Antoine Gournay and Antoine Derighetti for some nice discussions, and Damien Gaboriau for reading a preliminary version of this note.

\end{document}